\documentclass{amsart}
\usepackage{amsmath,amssymb,mathtools}
\usepackage{hyperref}
\usepackage{xcolor}

\theoremstyle{plain}
\newtheorem{theorem}{Theorem}[section]
\newtheorem{proposition}[theorem]{Proposition}
\newtheorem{lemma}[theorem]{Lemma}
\newtheorem{corollary}[theorem]{Corollary}
\numberwithin{equation}{section}
\theoremstyle{remark}
\newtheorem{remark}[theorem]{Remark}

\theoremstyle{definition}

\newcommand{\Z}{\mathbb Z}
\newcommand{\F}{\mathbb F}
\newcommand{\Q}{\mathbb Q}
\newcommand{\cZ}{\mathcal Z}
\newcommand{\RZ}{\mathbb R \mathcal Z}
\newcommand{\spn}[2]{\operatorname{span}_{\Z_#1}\{#2\}}
\newcommand{\cO}{\mathcal O}
\DeclareMathOperator{\conv}{conv}

\newcommand{\R}{\mathbb R}

\newcommand{\wideand}{\quad \text{and} \quad}

\author[S. Choi]{Suyoung Choi}
\address{Department of Mathematics, Ajou University, 206, World cup-ro, Yeongtong-gu, Suwon 16499, Republic of Korea}
\email{schoi@ajou.ac.kr}

\author[H. Jang]{Hyeontae Jang}
\address{School of Mathematics, Korea Institute for Advanced Study, 85 Hoegiro, Dongdaemun-gu, Seoul 02455, Republic of Korea}
\email{hjang1112@kias.re.kr}

\title[Separation of Buchstaber invariants]{A finite-geometric obstruction and a polytopal separation of Buchstaber invariants}
\date{\today}

\thanks{The first-named author was supported by the National Research Foundation of Korea Grant funded by the Korean Government (RS-2025-00521982). The second-named author was supported by a KIAS Individual Grant (MG105401) at Korea Institute for Advanced Study}

\keywords{Buchstaber number, real Buchstaber number, polytopal sphere, toric lifting problem, finite geometry, stellar subdivision}

\subjclass[2020]{Primary 57S12; Secondary 52B05, 05B35, 05B25}

\begin{document}
\maketitle

\begin{abstract}
    We construct a polytopal simplicial sphere that admits a mod $2$ characteristic map but no mod $3$ characteristic map, and hence no integral one.
    This shows that the Buchstaber number of a polytopal sphere can be strictly smaller than its real Buchstaber number, and gives a negative answer to the toric lifting problem.
\end{abstract}

\section{Introduction}
Let $K$ be an $(n-1)$-dimensional finite simplicial complex on the vertex set~$[m]=\{1,\ldots,m\}$. 
Its moment-angle complex and real moment-angle complex are
$$
    \cZ_K=\bigcup_{\sigma\in K}(D^2)^\sigma\times(S^1)^{[m]\setminus\sigma} \wideand \RZ_K=\bigcup_{\sigma\in K}(D^1)^\sigma\times(S^0)^{[m]\setminus\sigma}.
$$
They carry canonical actions of $T^m=(S^1)^m$ and $\Z_2^m$, respectively, given by coordinatewise rotations and sign changes. 
The \emph{Buchstaber number} $s(K)$ is the largest dimension of a subtorus of $T^m$ acting freely on $\cZ_K$.
The \emph{real Buchstaber number} $s_{\R}(K)$ is the largest rank of a subgroup of $\Z_2^m$ acting freely on~$\RZ_K$. 
These invariants measure the size of free subactions of the canonical actions; see~\cite{Fukukawa-Masuda2011, Erokhovets2014, Buchstaber-Panov2015book}.

The subgroup of elements of order dividing two in a freely acting $r$-dimensional subtorus induces a free $\Z_2^r$-action on $\RZ_K$. 
Together with the isotropy groups corresponding to maximal simplices, this gives
$$
    s(K)\le s_{\R}(K)\le m-n,
$$
as observed in~\cite{Fukukawa-Masuda2011}.

Assume now that $K$ is pure.
An \emph{integral characteristic map} assigns a vector~$\lambda(i)\in\Z^n$ to each vertex $i$ so that the vectors assigned to every facet form a basis of $\Z^n$. 
For a prime $p$, a \emph{mod~$p$ characteristic map} is defined in the same way, with $\Z^n$ replaced by $\Z_p^n$, where $\Z_p=\Z/p\Z$.
The extreme values of the two invariants are described by these maps:
\begin{align*}
    s(K)=m-n &\quad\Longleftrightarrow\quad K\text{ admits an integral characteristic map},\\
    s_{\R}(K)=m-n &\quad\Longleftrightarrow\quad K\text{ admits a mod }2\text{ characteristic map}.  
\end{align*}
These notions come from toric topology~\cite{Davis-Januszkiewicz1991, Ishida-Fukukawa-Masuda2013, Buchstaber-Panov2015book}.
If $K$ is the boundary of the polar of a simple $n$-polytope $P$, then integral characteristic maps on $K$ are the characteristic maps of quasitoric manifolds over $P$, and mod~$2$ characteristic maps are those of small covers over $P$~\cite{Davis-Januszkiewicz1991}. 
Reduction modulo~$2$ corresponds to passing from a quasitoric manifold to the fixed point set of its conjugation, which is a small cover over the same polytope. 
Thus every simple polytope that supports a quasitoric manifold also supports a small cover. 
A positive answer to either of the following two questions would give the converse.

The first is the \emph{toric lifting problem}, posed by Zhi L\"u in 2011 and documented in~\cite{Choi-Park2016Wedge}: is every mod~$2$ characteristic map on a PL sphere the mod~$2$ reduction of an integral characteristic map? 
For polytopal spheres, this asks whether every small cover is the fixed point set of the conjugation on a quasitoric manifold. 
Positive results have been obtained in several cases~\cite{Choi-Jang-Vallee2025JLMS, Choi-Guan-Jang-Lu2026}. 

The second question concerns the invariants themselves. 
It was shown that the two Buchstaber numbers differ for the real universal complex whose vertices are the nonzero vectors of $\Z_2^4$ and whose simplices are the linearly independent sets, see~\cite{Ayzenberg2014preprint,Sun2017}. 
On the other hand, $s(K)=s_\R(K)$ whenever $n\le3$~\cite[Proposition~3.2.35]{Erokhovets2014}. 
Ayzenberg~\cite[Problem~1]{Ayzenberg2014preprint} and Erokhovets~\cite[Problem~3.2.39]{Erokhovets2014} asked whether $s(K)=s_\R(K)$ for every polytopal sphere $K$. 
No PL~sphere with $s(K) < s_\R(K)$ was previously known.

We answer both questions negatively.

\begin{theorem}\label{thm:main}
    There exists a polytopal sphere $K$ that admits a mod~$2$ characteristic map but no mod~$3$ characteristic map. 
    Consequently, 
    $$
        s(K) < s_{\R}(K).
    $$
\end{theorem}

In particular, none of the mod~$2$ characteristic maps on $K$ lifts to an integral one, which gives a negative answer to the toric lifting problem. 
Passing to the polar polytope, we obtain the following.

\begin{corollary}\label{cor:polytope}
    There is a simple polytope that supports a small cover but no quasitoric manifold.
    In particular, there is a small cover that is not weakly equivariantly homeomorphic to the fixed point set of the conjugation on any quasitoric manifold.
\end{corollary}

\section{Finite-geometric obstructions} \label{sec:alg_lem}
Put $E=\Z_2^4\setminus\{0\}$.
To each $a\in E$, assign two vectors $u_a, v_a\in\Z_3^4$.
Sums of labels, such as $a+b$, are taken in $\Z_2^4$.

\begin{lemma}\label{lem:mixed}
    There are no two families $\{u_a\}_{a\in E}$ and $\{v_a\}_{a\in E}$ such that
    \begin{equation}\label{eq:mixed}
        \det(u_a,u_b,v_c,v_d)\ne0 \quad\text{whenever }(a,b,c,d)\text{ is a basis of }\Z_2^4.
    \end{equation}
\end{lemma}
\begin{proof}
Suppose, for a contradiction, that two such families exist.

Any two distinct labels $a,b\in E$ are linearly independent over $\Z_2$, so they extend to a basis $(a,b,c,d)$.
Applying~\eqref{eq:mixed} to $(a,b,c,d)$ and $(c,d,a,b)$ gives
$$
    \det(u_a,u_b,v_c,v_d)\ne0 \wideand \det(u_c,u_d,v_a,v_b)\ne0.
$$
Thus $u_a,u_b$ are linearly independent, and so are $v_a,v_b$.
In particular, all the vectors are nonzero, and no two vectors in the same family are scalar multiples of each other.

We may apply the same invertible linear map to all the vectors in both families.
We may also multiply any individual vector by a nonzero scalar.
Neither operation changes condition~\eqref{eq:mixed}.

We first prove that
\begin{equation}\label{eq:mixed-span}
    u_{a+b}\in \spn{3}{u_a,u_b} \qquad\text{for all distinct }a,b\in E.
\end{equation}
Suppose this fails for some $a,b$.
Then $u_a,u_b,u_{a+b}$ are linearly independent.
After a change of coordinates, we may assume that $u_a=e_1$, $u_b=e_2$, $u_{a+b}=e_3$, where $e_1,e_2,e_3,e_4$ are the standard basis vectors of $\Z_3^4$.

Let $L=\spn{2}{a,b}$.
The twelve labels outside $L$ form three cosets $C_1,C_2,C_3$ of $L$, each containing four labels.
If $x$ and $y$ belong to different cosets, their images in~$\Z_2^4/L\cong\Z_2^2$ are distinct and nonzero.
They are therefore linearly independent in this quotient.
It follows that each of $(a,b,x,y)$, $(a,a+b,x,y)$, $(b,a+b,x,y)$ is a basis of $\Z_2^4$.

Write $v_x=(z_{x,1},z_{x,2},z_{x,3},w_x)$.
Apply~\eqref{eq:mixed} to the three bases above.
Expanding the determinants along the first two columns gives
\begin{equation}\label{eq:ratios}
    z_{x,k}w_y-z_{y,k}w_x\ne0  \qquad(k=1,2,3)
\end{equation}
whenever $x$ and $y$ belong to different cosets.

If two different cosets each contained a label with $w_x=0$, we could choose $x,y$ from those cosets with $w_x=w_y=0$.
This would contradict~\eqref{eq:ratios}.
Hence at least two cosets contain no label with $w_x=0$.
Call them $C_1$ and $C_2$.

For each $x\in C_1\cup C_2$, multiply $v_x$ by $w_x^{-1}$.
The last coordinate of each of these eight vectors is now $1$.
Let $S,T\subset\Z_3^3$ be the sets of their first three coordinates, for $x\in C_1$ and $x\in C_2$, respectively.
Each set has four distinct points, since no two vectors in the $v$-family are scalar multiples of each other.

For $k=1,2,3$, let $S_k$ and $T_k$ be the sets of values appearing in the $k$th coordinate of $S$ and $T$.
Since all the last coordinates are now $1$, \eqref{eq:ratios} says that $S_k\cap T_k=\varnothing$.
These are nonempty subsets of the three-element set $\Z_3$, so for each $k=1,2,3$, at least one of $S_k$ and $T_k$ is a singleton.
Since there are three coordinates, after interchanging $S$ and $T$ and permuting the coordinates if necessary, we may assume $|S_1| = |S_2| = 1$.
Since $T_3$ is nonempty and disjoint from $S_3$, we have $|S_3| \leq 2$.
Thus $|S| \leq 2$, contrary to $|S| = 4$.
This contradiction proves~\eqref{eq:mixed-span}.

We now use~\eqref{eq:mixed-span} to obtain the final contradiction.
Choose three linearly independent labels $a,b,c\in\Z_2^4$.
The seven vectors with labels
$$
    a,\ b,\ c,\ a+b,\ a+c,\ b+c,\ a+b+c
$$
all lie in $ W=\spn{3}{u_a,u_b,u_c}$.
Indeed, this follows directly from~\eqref{eq:mixed-span} for the first six vectors.
For the last one, apply the same relation to the labels $a+b$ and $c$.

We claim that $\dim W=3$.
Otherwise, these seven vectors would lie in a vector space of dimension at most two over $\Z_3$.
A two-dimensional vector space over $\Z_3$ has eight nonzero vectors, which form four pairs $\{w,-w\}$.
It therefore contains at most four nonzero vectors no two of which are scalar multiples of each other.
This contradicts the existence of our seven vectors.

The same argument applies to every basis of~$U=\spn{2}{a,b,c}$.
By~\eqref{eq:mixed-span}, $u_p,u_q,u_{p+q}$ are dependent for any distinct $p, q \in U \setminus \{0\}$.
All other triples of the seven vectors indexed by $U$ are independent by the preceding argument.
Hence the seven vectors represent the Fano matroid over~$\Z_3$.
This is impossible over~$\Z_3$, since the Fano matroid is representable only over fields of characteristic two~\cite[Proposition~6.4.8(i)]{Oxley2011book}.
\end{proof}

\begin{remark}
  Taking $u_a = v_a$ for all $a \in E$ in Lemma~\ref{lem:mixed} shows that the real universal complex associated with $\Z_2^4$ admits no mod~$3$ characteristic map~\cite{Baralic-Vavpetic-Vucic2024}.
  It therefore admits no integral characteristic map either~\cite{Ayzenberg2014preprint}.
  In fact, its Buchstaber number is known to be~$10$~\cite{Sun2017}.
\end{remark}

We next turn Lemma~\ref{lem:mixed} into conditions on the minors of a single matrix.
For a matrix $D$ and a set $\sigma$ of column indices, write $D_\sigma$ for the submatrix with those columns in increasing order.

List the elements of $E$ as $b_1,\ldots,b_{15}$, and put
\begin{equation}\label{eq:parameters}
    N=3\cdot41^{30}+1 \wideand m_0=n+4=30N.
\end{equation}
Define a $4\times m_0$ matrix $B$ over $\Z_2$ by repeating the block of $30$ columns
$$
    (b_1,b_1,b_2,b_2,\ldots,b_{15},b_{15})
$$
exactly $N$ times.
Equivalently,
\begin{equation}\label{eq:binary}
    B_{30(t-1)+2j-1}=B_{30(t-1)+2j}=b_j \qquad(1\le t\le N,\ 1\le j\le15).
\end{equation}
Every label occurs at both an odd and an even position in each block. Since a block has even length, these are also the parities of the indices in the full matrix.

Let $\cO$ consist of the sets $\sigma=\{i_1<i_2<i_3<i_4\}\subset[m_0]$ such that $i_1,i_2,i_3,i_4$ have alternating parity and $\det B_\sigma = 1 \in \Z_2$.
It is clear that $\cO$ is nonempty.

\begin{lemma}\label{lem:repeat}
    There is no matrix $A\in M_{4,m_0}(\Z_3)$ such that
    $$
        \det A_\sigma\ne0 \in \Z_3 \qquad\text{for every }\sigma\in\cO.
    $$
\end{lemma}
\begin{proof}
    Suppose that such a matrix $A$ exists.
    Give two nonzero columns the same color exactly when they are scalar multiples of each other.
    Give the zero vector a separate color.
    There are $3^4-1=80$ nonzero vectors in $\Z_3^4$, and each nonzero color consists of a pair $\{w,-w\}$.
    Thus there are $41$ colors in total.
    
    Divide $A$ into $N$ consecutive blocks of $30$ columns.
    A block has at most $41^{30}$ possible ordered color lists.
    Since $N=3\cdot41^{30}+1$, four blocks have the same list.
    Write their positions as $t_1<t_2<t_3<t_4$.
    
    For each $b_j\in E$, let $u_{b_j}$ represent the common color at position $2j-1$ in these four blocks, and let $v_{b_j}$ represent the common color at position $2j$.
    Represent the zero color by zero.
    
    Take any ordered binary basis $(a,b,c,d)$.
    In the four blocks, select columns as follows:
    \[
    \begin{array}{c|c|c}
        \text{block} & \text{binary label} & \text{position}\\ \hline
        t_1 & a & \text{odd}\\
        t_2 & c & \text{even}\\
        t_3 & b & \text{odd}\\
        t_4 & d & \text{even}
    \end{array}
    \]
    Their column indices in $A$ are increasing and have alternating parity, and their binary labels form a basis.
    Their index set therefore belongs to $\cO$, so these columns are independent over $\Z_3$.
    
    The selected columns have the same colors as $u_a,v_c,u_b,v_d$, respectively.
    They are nonzero, so replacing them by these representatives preserves independence. 
    Interchanging the middle two columns now gives
    $$
        \det(u_a,u_b,v_c,v_d)\ne0.
    $$
    This holds for every binary basis $(a,b,c,d)$, contrary to Lemma~\ref{lem:mixed}.
\end{proof}

We will use the following standard relation between complementary minors; see~\cite[Theorem~2.2.8]{Oxley2011book}.
\begin{lemma}\label{lem:complementary-minors}
    Let $D\in M_{r,q}(\F)$ have rank $r$, where $\F$ is a field and $r<q$.
    Let the rows of $H\in M_{q-r,q}(\F)$ form a basis of $\ker D$.
    For every $r$-element set $\sigma\subset[q]$, we have
    $$
        \det D_\sigma\ne0 \quad\Longleftrightarrow\quad \det H_{\sigma^c}\ne0,
    $$
    where $\sigma^c=[q]\setminus\sigma$.
\end{lemma}

We now extend an assignment of nonzero binary vectors to a mod~$2$ characteristic map after stellar subdivisions.
Facets on which the assigned vectors are already independent will be left unchanged.
The proof of the following lemma uses a determinant reduction similar to that in~\cite[Theorem~2.1]{Bruns-vonThaden2017}.

\begin{lemma}\label{lem:relative-subdivision}
    Let $L$ be a finite pure $(d-1)$-dimensional simplicial complex on~$[q]$, where $d\ge1$, and let $\lambda\colon[q]\to\Z_2^d\setminus\{0\}$.
    Then there exists a finite pure $(d-1)$-dimensional simplicial complex $L'$ with the following properties:
    \begin{enumerate}
      \item $\lambda$ extends to a mod~$2$ characteristic map on $L'$.
      \item Every facet of $L$ with $\det \lambda_F = 1$ remains a facet of $L'$.
      \item $L'$ is polytopal whenever $L$ is polytopal.
    \end{enumerate}
\end{lemma}
\begin{proof}
    Choose integer lifts $V_i$ of $\lambda(i)$ so that $\det V_F\ne0$ for every facet $F$.
    Indeed, fix any integer representatives $\widetilde{\lambda(i)}$ and write $V_i = \widetilde{\lambda(i)} + 2 X_i$.
    Each facet determinant is a nonzero polynomial in the entries of $X = (X_1, \ldots, X_q)$, since its degree-$d$ part is $2^d \det X_F$.
    The product over all facets is therefore a nonzero polynomial, so it is nonzero at some integer matrix $X$.
    
    If an even determinant remains, choose a facet $F$ with the largest even absolute determinant.
    Let $\tau\subseteq F$ be a minimal dependent set modulo two.
    Then $|\tau|\ge2$, and we can write
    $$
        \sum_{i\in\tau}V_i=2^rW, \quad \text{ for some } r\ge1 \text{ and } W\not\equiv0\pmod2.
    $$
    The sum is nonzero because the vectors on $F$ are independent over $\Q$.
    
    Stellarly subdivide at $\tau$ and assign $W$ to the new vertex $v$.
    For each facet $G$ containing $\tau$, the new facets are $G_i=(G\setminus\{i\})\cup\{v\}$, where $i\in\tau$.
    Expansion in the new column gives
    $$
        |\det V_{G_i}|=2^{-r}|\det V_G|<|\det V_G|.
    $$
    
    Let $M = | \det V_F |$.
    Since $\tau$ is dependent modulo~$2$, every facet $G$ containing~$\tau$ has even determinant with $| \det V_G | \leq M$.
    Hence each new facet has absolute determinant at most $M/2 < M$, while all facets with odd determinant remain unchanged.
    Each step therefore decreases either the largest even absolute determinant or the number of facets attaining it.
    The process stops after finitely many steps.
    
    Reducing the final vectors modulo two gives the required mod~$2$ characteristic map.
    No original vertex is removed, since every subdivision is at a face with at least two vertices.
    Polytopality is preserved under stellar subdivisions~\cite{Ewald-Shephard1974}.
\end{proof}

\section{Construction of the sphere}\label{sec:construction}

We use the matrix $B$, the collection $\cO$, and the integers $m_0,n$ defined in Section~\ref{sec:alg_lem}.
We first construct a simplicial polytope in which every complement of a set in $\cO$ is a facet.
We then apply Lemma~\ref{lem:relative-subdivision} to obtain the required polytopal sphere.

Put $p_j=(j,j^2,\ldots,j^n)\in\R^n$ for $j \in [m_0]$, and let
$$
    P=\conv(p_1,\ldots,p_{m_0}).
$$
This is known as a \emph{cyclic polytope}.
It is an $n$-dimensional simplicial polytope, and each~$p_j$ is a vertex of $P$.
\begin{lemma}\label{lem:cyclic-facets}
    For every $\sigma \in \cO$, the set $[m_0]\setminus\sigma$ is the vertex set of a facet of $P$.
\end{lemma}
\begin{proof}
    By Gale's evenness criterion~\cite{Ziegler1995book}, the complement of a four-element set is a facet of $P$ if and only if the four indices, in increasing order, have alternating parity.
    Every set in~$\cO$ satisfies this condition.
\end{proof}

We now construct the required polytopal sphere.
\begin{proposition}\label{prop:polytope}
    There is a polytopal $(n-1)$-sphere $K$ on the vertex set $[m]$, for some $m\ge m_0$, such that:
    \begin{enumerate}
        \item For every $\sigma\in\cO$, the set $[m_0]\setminus\sigma$ is a facet of $K$ with its original labels.
        \item The sphere $K$ admits a mod~$2$ characteristic map.
    \end{enumerate}
\end{proposition}
\begin{proof}
    The matrix $B$ has rank $4$, because its columns include every nonzero vector of $\Z_2^4$.
    Its kernel therefore has dimension $m_0-4=n$.
    Choose an $n\times m_0$ matrix $\Lambda$ over $\Z_2$ whose rows form a basis of $\ker B$.
    Lemma~\ref{lem:complementary-minors} gives 
    \begin{equation}\label{eq:complement}
        \det\Lambda_{\sigma^c}\ne0     \qquad\text{for every }\sigma\in\cO.
    \end{equation}
    
    Every column of $\Lambda$ is nonzero.
    Indeed, for each $j$ there is another index $k$ with $B_j=B_k$, so $e_j+e_k\in\ker B$.
    If the $j$th column of $\Lambda$ were zero, every vector in its row space would have zero $j$th coordinate.
    This contradicts $e_j+e_k\in\ker B$.
    
    Let $L$ be the boundary complex of the cyclic polytope $P$ defined above.
    Assign $\Lambda_j$ to vertex $j$.
    Apply Lemma~\ref{lem:relative-subdivision} to this assignment.
    We obtain a polytopal $(n-1)$-sphere $K$ with a mod~$2$ characteristic map, which proves (2).

    No original vertex is removed.
    Keep the labels $1,\ldots,m_0$ for the original vertices, and label the new vertices by $m_0+1,\ldots,m$.
    For every $\sigma \in \cO$, Lemma~\ref{lem:cyclic-facets} says that $\sigma^c = [m_0] \setminus \sigma$ is a facet of $L$.
    By~\eqref{eq:complement}, the columns of $\Lambda$ on this facet are independent.
    Lemma~\ref{lem:relative-subdivision} therefore leaves this facet unchanged, which proves (1).
\end{proof}

\begin{proof}[Proof of Theorem~\ref{thm:main}]
    Let $K$ be the polytopal sphere from Proposition~\ref{prop:polytope}.
    By Proposition~\ref{prop:polytope}(2), the sphere $K$ admits a mod~$2$ characteristic map.
    
    Suppose, for a contradiction, that $K$ also has a mod~$3$ characteristic map.
    Retain only its columns with the original labels $1,\ldots,m_0$, and denote by $\Omega$ the resulting $n\times m_0$ matrix over $\Z_3$.
    For every $\sigma\in\cO$, Proposition~\ref{prop:polytope}(1) says that $\sigma^c$ is a facet of $K$ with only original labels.
    Thus $\det\Omega_{\sigma^c}\ne0$ for every $\sigma\in\cO$.
    Since $\cO$ is nonempty, $\Omega$ has rank $n$.
    Its kernel has dimension $m_0-n=4$.
    Choose $A\in M_{4,m_0}(\Z_3)$ whose rows form a basis of this kernel.
    Applying Lemma~\ref{lem:complementary-minors} to $\Omega$ and the $n$-element set $\sigma^c$ gives
    $$
        \det A_\sigma\ne0 \qquad\text{for every }\sigma\in\cO.
    $$
    This contradicts Lemma~\ref{lem:repeat}.
    Therefore $K$ has no mod~$3$ characteristic map.
    
    Any integral characteristic map would reduce modulo~$3$ to a mod~$3$ characteristic map.
    Thus $K$ has no integral characteristic map, and therefore
    $$
        s(K)<m-n=s_{\R}(K).
    $$
\end{proof}

\end{document}